\documentclass[11pt]{article}
\usepackage{amsmath}
\usepackage{amssymb}
\usepackage{amsthm}
\usepackage{pstricks}
\usepackage{pstricks-add}
\usepackage{fullpage}

\newtheorem{theorem}{Theorem}
\newtheorem{lemma}{Lemma}

\newtheorem{definition}{Definition}
\newtheorem{notation}{Notation}
\newtheorem{remark}{Remark}

\newcommand{\tdeg}{\mathrm{tdeg}}
\newcommand{\s}[1]{\mbox{\sf #1}}

\newenvironment{glists}[4]{
                 \begin{list}{}{
                     \setlength{\labelwidth}{#2}
                     \setlength{\labelsep}{#3}
                     \setlength{\leftmargin}{#1}
                     \addtolength{\leftmargin}{\labelwidth}
                     \addtolength{\leftmargin}{\labelsep}
                     \setlength{\parsep}{#4}
                     \setlength{\topsep}{\parsep}
                     \setlength{\itemsep}{\parsep}
                     \setlength{\listparindent}{0in}
                     }
                 }{
                 \end{list}
                 }
\newenvironment{clists}{
\begin{glists}{0em}{3.5em}{0.5em}{0.5em}
}{
\end{glists}
}

\newcommand{\citem}[1]{\item[\sf #1\hfill]}

\catcode`\@=11
\newbox\bracku
\setbox\bracku=\hbox{\vrule height 4pt depth 0pt width 0.1pt}
\ht\bracku=1pt\dp\bracku=0pt
\def\upbrackfill{$\m@th\copy\bracku\leaders\hrule\hfill\copy\bracku$}
\def\underbrack#1{
  \mathop{
    \vtop{
      \ialign{
        ##\crcr$\hfil\displaystyle{\kern1pt#1\kern1pt}
        \hfil$\crcr\noalign{\kern3\p@\nointerlineskip}
        \upbrackfill\crcr\noalign{\kern3\p@}
      }
    }
  }
  \limits
}
\catcode`\@12

\begin{document}

\title{Maximum  Gap in (Inverse) Cyclotomic Polynomial}

\author{Hoon Hong, Eunjeong Lee, Hyang-Sook Lee and Cheol-Min Park}

\maketitle

\begin{abstract}
Let $g(f)$
denote the maximum of the differences (gaps) between two consecutive exponents
occurring in  a polynomial $f$.
Let $\Phi_n$ denote the $n$-th cyclotomic polynomial and let $\Psi_n$ denote the $n$-th inverse cyclotomic polynomial. 
In this note, we study $g(\Phi_n)$ and $g(\Psi_n)$ where $n$ is a product of odd primes, say $p_1 < p_2 < p_3$, etc. 
It is trivial to determine  $g(\Phi_{p_1})$, $g(\Psi_{p_1})$ and $g(\Psi_{p_1p_2})$. 
Hence the simplest non-trivial cases are $g(\Phi_{p_1p_2})$ and $g(\Psi_{p_1p_2p_3})$.
We provide an exact expression for $g(\Phi_{p_1p_2}).$ We also provide an exact expression for 
 $g(\Psi_{p_1p_2p_3})$ under a mild condition. The condition is almost always satisfied (only finite exceptions for each $p_1$). We also provide a lower bound and an upper bound for $g(\Psi_{p_1p_2p_3})$.
\end{abstract}

\section{Introduction}

The $n$-th cyclotomic polynomial $\Phi_n$ and the $n$-th inverse cyclotomic polynomial $\Psi_n$ are defined by
\[
\Phi_n(x) \;=\; \prod_{\substack{1\leq j\leq n\\(j,n)=1}} (x-\zeta_n^j)\;\;\;\;\;\;\;\;\;\;\;\;\;\;\;
\Psi_n(x) \;=\; \prod_{\substack{1\leq j\leq n\\(j,n)>1}} (x-\zeta_n^j)
\]
where $\zeta_n$ is a primitive $n$-th root of unity.
For example,
we have
\begin{eqnarray*}
\Phi_{3}(x)          &=& 1+x+{x}^{2}\\
\Phi_{3\cdot 5}(x)   &=& 1-x+{x}^{3}-{x}^{4}+{x}^{5}-{x}^{7}+{x}^{8}\\
\Psi_{3}(x)          &=& -1+x\\
\Psi_{3\cdot 5}(x)   &=& -1-x-{x}^{2}+{x}^{5}+{x}^{6}+{x}^{7}
\\
\Psi_{3\cdot 5\cdot 7}(x) &=&  -1+x-{x}^{3}+{x}^{4}-{x}^{5}-{x}^{10}+{x}^{11}-{x}^{12}-{x}^{17}+{x}^{
18}-{x}^{19}+{x}^{21}-{x}^{22}\\
&&+{x}^{35}-{x}^{36}+{x}^{38}-{x}^{39}+{x}
^{40}+{x}^{45}-{x}^{46}+{x}^{47}+{x}^{52}-{x}^{53}+{x}^{54}-{x}^{56}+{
x}^{57}
\end{eqnarray*}

There have been extensive studies on the coefficients of cyclotomic polynomials~\cite{
Bachman2003,%
Beiter1978,%
Endo1974,%
JiLi2008,%
Lehmer1936,%
JiLiMor09,%
Suzuki1987,%
Thangadurai1999}.
Recently there have been also studies on the coefficients of inverse cyclotomic polynomials~\cite{
Mor09,%
Bzdega10%
}.
In this note, we study the exponents  of (inverse) cyclotomic polynomials. In particular, we are interested in  the {\em maximum gap}, $g(f)$, which is the maximum of the differences (gaps) between two consecutive exponents
occurring in~$f$ where $f=\Phi_n$ or $f=\Psi_n$.
More precisely the maximum gap is defined as follows:
\begin{definition}[Maximum Gap]\label{def:maxgap}
Let $f = c_1x^{e_1}+\cdots c_tx^{e_t}$ 
where $c_1,\ldots,c_t \neq 0$ and $e_1 < \cdots < e_t$. 
Then the maximum gap of $f$, written as $g(f)$, is defined
by
\[
g(f) = \max_{1\leq i < t} (e_{i+1}-e_i),\;\;\;\; g(f)=0\;\;\mbox{when } \;t=1
\]
\end{definition}
\noindent For example, $g(\Phi_{3\cdot5})=2$ because $2$ is the maximum among   $1-0,\;3-1,\;4-3,\;5-4,\;7-5,\;8-7$.

It can be visualized by the following diagrams where a long bar represents a polynomial.
The black color indicates that the corresponding exponent (term) occurs in the polynomial
and the white color indicates that it does not.

{
\tiny
\psset{unit=0.25}
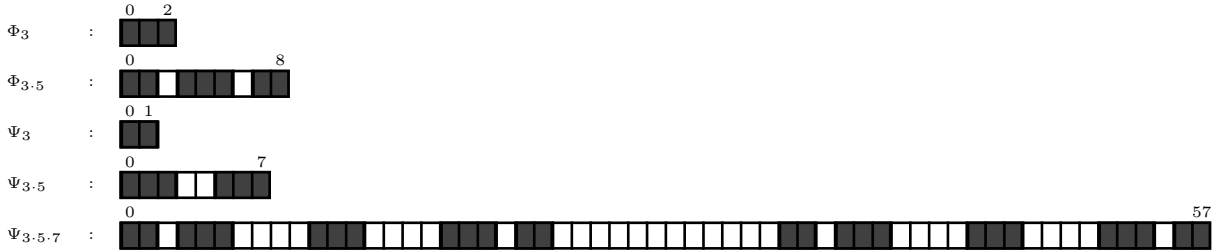
\[
\begin{array}{lll}
\Phi_{3} &:& \begin{pspicture}(0.0,0.6)(58,1.5)
\rput(0.5,2.0){0}
\rput(2.5,2.0){2}
\psframe[linewidth=1pt](0.0,0)(3,1.5)
\psframe[linewidth=1pt,fillstyle=solid,fillcolor=darkgray](0,0)(3,1.5)
\multips(0,0)(1,0){3}{\psline[linewidth=1pt](0,0)(0,1.5)}
\end{pspicture}\\
\Phi_{3\cdot5} &:& \begin{pspicture}(0.0,0.6)(58,3.0)
\rput(0.5,2.0){0}
\rput(8.5,2.0){8}
\psframe[linewidth=1pt](0.0,0)(9,1.5)
\psframe[linewidth=1pt,fillstyle=solid,fillcolor=darkgray](0,0)(2,1.5)
\psframe[linewidth=1pt,fillstyle=solid,fillcolor=darkgray](3,0)(6,1.5)
\psframe[linewidth=1pt,fillstyle=solid,fillcolor=darkgray](7,0)(9,1.5)
\multips(0,0)(1,0){9}{\psline[linewidth=1pt](0,0)(0,1.5)}
\end{pspicture}\\
\Psi_{3} &:&
\begin{pspicture}(0.0,0.6)(58,3.0)
\rput(0.5,2.0){0}
\rput(1.5,2.0){1}
\psframe[linewidth=1pt](0.0,0)(2,1.5)
\psframe[linewidth=1pt,fillstyle=solid,fillcolor=darkgray](0,0)(2,1.5)
\multips(0,0)(1,0){3}{\psline[linewidth=1pt](0,0)(0,1.5)}
\end{pspicture}\\
\Psi_{3\cdot 5} &:&
\begin{pspicture}(0.0,0.6)(58,3.0)
\rput(0.5,2.0){0}
\rput(7.5,2.0){7}
\psframe[linewidth= 1pt](0.0,0)(8,1.5)
\psframe[linewidth=1pt,fillstyle=solid,fillcolor=darkgray](0,0)(3,1.5)
\psframe[linewidth=1pt,fillstyle=solid,fillcolor=darkgray](5,0)(8,1.5)
\multips(0,0)(1,0){8}{\psline[linewidth=1pt](0,0)(0,1.5)}
\end{pspicture}\\
\Psi_{3\cdot 5 \cdot 7} &:&
\begin{pspicture}(0.0,0.6)(58,3.0)
\rput(0.5, 2.0){0}
\rput(57.5,2.0){57}
\psframe[linewidth=1pt](0.0,0)(58.0,1.5)
        \psframe[linewidth=1pt,fillstyle=solid,fillcolor=darkgray](0,0)(2,1.5)
\psframe[linewidth=1pt,fillstyle=solid,fillcolor=darkgray](3,0)(6,1.5)
\psframe[linewidth=1pt,fillstyle=solid,fillcolor=darkgray](10,0)(13,1.5)
\psframe[linewidth=1pt,fillstyle=solid,fillcolor=darkgray](17,0)(20,1.5)
\psframe[linewidth=1pt,fillstyle=solid,fillcolor=darkgray](21,0)(23,1.5)
\psframe[linewidth=1pt,fillstyle=solid,fillcolor=darkgray](35,0)(37,1.5)
\psframe[linewidth=1pt,fillstyle=solid,fillcolor=darkgray](38,0)(41,1.5)
\psframe[linewidth=1pt,fillstyle=solid,fillcolor=darkgray](45,0)(48,1.5)
\psframe[linewidth=1pt,fillstyle=solid,fillcolor=darkgray](52,0)(55,1.5)
\psframe[linewidth=1pt,fillstyle=solid,fillcolor=darkgray](56,0)(58,1.5)
\multips(0,0)(1,0){58}{\psline[linewidth=1pt](0,0)(0,1.5)}
\end{pspicture}
\end{array}
\]
}

\noindent 
One immediately notices that the maximum gap is essentially the length of a longest 
white block  {\em plus\/}~$1$. 
For example, a longest white block in $\Phi_{3\cdot5}$ has length 1. Hence~$g(\Phi_{3\cdot5})=1+1=2$.

Our initial motivation came from its need for analyzing the complexity~\cite{HLLP11a}
of a certain paring operation over elliptic curves~\cite{
FST,%
LLP,%
ZZH%
}.  
However, it seems to be  a curious problem on its own and it could be also viewed as a first step
toward the detailed understanding of the sparsity structure of~$\Phi_n$ and~$\Psi_n$.

In this note, we tackle the simplest non-trivial cases, namely,  
$g(\Phi_{p_1p_2})$ and $g(\Psi_{p_1p_2p_3})$ where $p_1 <p_2<p_3 $ are odd primes. 
As far as we are aware,  there were  no  published results on this problem. We will provide an exact expression for $g(\Phi_{p_1p_2})$ in Theorem~\ref{theorem:c2p}. 
We will also provide an exact expression for $g(\Psi_{p_1p_2p_3})$ under a mild condition  in Theorem~\ref{theorem:ic3p}.
In Remark~\ref{rem:freq} we will show that the condition is very mild. Finally we will provide a lower bound and an upper bound for $g(\Psi_{p_1p_2p_3})$ in Theorem~\ref{theorem:icbound}.  

In order to obtain the results, we had to overcome a few difficulties. 
It can be easily shown that $\Phi_{p_1p_2}$ and  $\Psi_{p_1p_2p_3}$ are sums and products of 
simple polynomials with trivial gap structures.  However adding and multiplying them could
introduce new gaps,  eliminate existing gaps or change the sizes of existing gaps 
etc, in intricate manners, via accumulation or cancellation of terms, making the analysis very challenging. 
We overcame the obstacles in two ways: (1) find mild conditions on $p_1,p_2,p_3$ 
that ensure that accumulation or cancellation do not occur.
(2) find mild conditions  that allow us to bound the sizes of  gaps arising 
from accumulation or cancellation and show that such gaps cannot be the maximum gap. 

This note is structured as follows. In the following section (Section~\ref{sec:trivial}), we will  quickly take care of trivial cases,
in order to identify the simplest non-trivial cases to tackle. A reader can safely skip over this section. In the subsequent section (Section~\ref{sec:main}), we will provide the main results on the simplest non-trivial cases. In the final section, we will prove the main results (Section~\ref{sec:proof}).

\section{Trivial Cases}\label{sec:trivial}
In this section, we will  quickly take care of trivial cases,
in order to identify the simplest non-trivial cases that will be  tackled in the next section. A reader can safely skip over this section. In the following we will use basic properties of (inverse) cyclotomic polynomials without explicit references. The basic properties of cyclotomic polynomials  can be found in any standard textbooks. The basic properties of inverse cyclotomic polynomials can be found in Lemma 2 of~\cite{Mor09}. 

\begin{itemize}
\item Since 
\[\Phi_n(x) = \Phi_{\hat{n}}(x^{\frac{n}{\hat{n}}}) \;\;\;\;\;\;\;\;
\Psi_n(x) = \Psi_{\hat{n}}(x^{\frac{n}{\hat{n}}})\]
we  immediately have 
$$g(\Phi_{n}) = \frac{n}{\hat{n}}g(\Phi_{\hat{n}}),\;\;\;\;\;\;\;g(\Psi_{n}) = \frac{n}{\hat{n}}g(\Psi_{\hat{n}})$$ 
where $\hat{n}$ is the radical of $n$. Thus we will, without losing generality, restrict $n$ to be squarefree. 

\item Since 
\[\Phi_{2n}(x)=\pm\Phi_n(-x) \;\;\;\;\;\;\;\;\;\Psi_{2n}(x) = \pm(1-x^n)\Psi_n(-x)\]
for odd $n$, we immediately have 
\[
g(\Phi_{2n})=g(\Phi_n)\;\;\;\;\;\;\;\;\;\;\;\;\;\;\;\;g(\Psi_{2n})= \max\{g(\Psi_n),\deg(\Phi_n)\}
\]
Thus we will, without losing generality, further restrict $n$ to be  squarefree and odd, that is, a product of zero or more distinct odd primes. 

\item Consider the case when $n$ is a product of {\em zero }odd primes, that is $n=1$. Since 
\[\Phi_1(x)=-1+x \;\;\;\;\;\;\;\Psi_1(x)=1\] 
we have
\[
g(\Phi_1)=1\;\;\;\;\;\;\;\;\;\;\;\;\;\;\;g(\Psi_1)=0
\]

\item Consider the case when $n$ is a product of {\em one }odd primes, that is $n=p_1$. Since 

\[\Phi_{p_1}(x)=1+x+\cdots+x^{p_1-1}\;\;\;\;\;\;\;\;\Psi_{p_1}(x) = -1+x\] 
we   have  
\[
g(\Phi_{p_1})=1\;\;\;\;\;\;\;\;\;\;\;g(\Psi_{p_1})=1
\]

\item Consider the case when $n$ is a product of {\em two }odd primes, that is $n=p_1p_2$ where $p_1<p_2$.  Since 
\[
\Psi_{p_1p_2}(x)=-(1+x+\cdots+x^{p_1-1})+(x^{p_2}+x^{p_2+1}+\cdots+x^{p_2+p_1-1})
\] 
we  have 
\[
 g(\Psi_{p_1p_2})=p_2 - (p_1 - 1)
\]
\end{itemize}
Hence the simplest non-trivial cases are $g(\Phi_{p_1p_2})$ and $g(\Psi_{p_1p_2p_3})$. We will tackle these cases in the following section.

\section{Main Results}\label{sec:main}
In this section, 
we tackle the simplest non-trivial cases identified in the previous section. In particular, we provide an exact expression for $g(\Phi_{p_1p_2})$ in Theorem~\ref{theorem:c2p}. 
We also provide an exact expression for $g(\Psi_{p_1p_2p_3})$ under a mild condition  in Theorem~\ref{theorem:ic3p}.
In Remark~\ref{rem:freq} we show that the condition is very mild. Finally we provide a lower bound and an upper bound for $g(\Psi_{p_1p_2p_3})$ in Theorem~\ref{theorem:icbound}. 

\begin{theorem}\label{theorem:c2p}
Let $n=p_1p_2$ where $p_1<p_2$ are odd primes. Then we have
\[
g(\Phi_n) \;\;=\;\;p_1-1
\]
\end{theorem}

\begin{theorem}\label{theorem:ic3p}
Let $n=p_1p_2p_3$ where $p_1<p_2<p_3$ are odd primes  satisfying the condition:
\begin{equation}\label{eq:cond}
p_2 \;\; \geq \;\; 4(p_1-1)\;\;\;\; \mbox{or} \;\;\;\; p_3 \;\;\geq \;\;p_{1}^2
\end{equation}
 Then we have
\begin{equation*}\label{eq:exact_gap_formula_lambda}
g(\Psi_n) \;\;=\;\; 2n\;\frac{1}{p_1} -\deg(\Psi_n)
\end{equation*}

\end{theorem}
\begin{theorem}\label{theorem:icbound}
Let $n=p_1p_2p_3$ where $p_1<p_2<p_3$ are odd primes. Then we have
\[
\max \{\;p_1-1,\; 2n\;\frac{1}{p_1}-\deg(\Psi_n)\;\}\;\;\;\; \leq \;\;\;\; g(\Psi_n) \;\;\;\;<\;\;\;\; 2n\;\left(\frac{1}{p_1}+\frac{1}{p_2}+\frac{1}{p_3}\right)-\deg(\Psi_n)
\]
\end{theorem}
\begin{remark}\label{rem:freq}\rm We make several remarks.
\begin{itemize} 
\item Note that the condition~(\ref{eq:cond}) in Theorem~\ref{theorem:ic3p}  is ``almost always'' satisfied. Thus we ``almost always'' have  $$g(\Psi_n) \;\;=\;\; 2n\;\frac{1}{p_1}-\deg(\Psi_n)$$
More precisely,
for each $p_1,$ only finitely many out of infinitely many  $(p_2,p_3)$ violate the condition~(\ref{eq:cond}).   

\item Let $V_{p_1}$ be the finite set of $(p_2,p_3)$ violating the condition~(\ref{eq:cond}).  
For several small $p_1$ values and for every $(p_2,p_3) \in V_{p_1},$ we carried out  direct calculation of $g(\Psi_n),$ 
 obtaining the following frequency table 

\begin{center}
\begin{tabular}{|c|c|c|c|c|}
\hline
$p_1$ & $\# V_{p_1}  \s{}$ & $\# V_{p_1}^{(1)}$& $\# V_{p_1}^{(2)}$& $\# V_{p_1}^{(3)}$ \\
\hline
  3   &    1    & 1     &   0  &   0 \\
  5   &    12   & 12    &   0  &   0 \\
  7   &    40   & 39    &   0  &   1 \\
  11  &    147  & 137   &   9 &   1 \\
  13  &    252  & 244   &   6  &   2 \\
  17  &    528  & 504   &   23 &   1 \\
  19  &    690  & 671   &   18 &   1 \\
  23  &    1155 & 1126  &   27 &   2 \\
  \hline
\end{tabular}
\end{center}
where
\begin{eqnarray*}
 V_{p_1}^{(1)} &=& \{(p_2,p_3) \in V_{p_1}\; :\; g(\Psi_n)=2n\frac{1}{p_1}-\deg(\Psi_n)\}\\
 V_{p_1}^{(2)} &=& \{(p_2,p_3) \in V_{p_1}\; :\;g(\Psi_n)=p_1-1\}\\
 V_{p_1}^{(3)} &=& V_{p_1}  \; - \; \left(V_{p_1}^{(1)} \; \cup \; V_{p_1}^{(2)}\right)
\end{eqnarray*}
\item
The table suggests  that even among the finite set $V_{p_1}$, we have  almost always
\[g(\Psi_n) \;\;=\;\; 2n\;\frac{1}{p_1}-\deg(\Psi_n)\]
and sometimes
\[g(\Psi_n) \;\;=\;\;p_1 -1\]
and very rarely 
\[
g(\Psi_n) > \max \{\;p_1-1,\; 2n\;\frac{1}{p_1}-\deg(\Psi_n)\;\}
\] 
\item In fact, when $p_1=3$ or $5$, the table shows that \[g(\Psi_n) \;\;=\;\; 2n\;\frac{1}{p_1}-\deg(\Psi_n)\]
\item  It is important to recall that for each $p_1$, for instance $p_1=23$, 
there are infinitely many possible values for  $(p_2,p_3)$. 
The table shows that for those {\em infinitely many\/} possible values of $(p_2,p_3)$,
the maximum gap is exactly the lower bound in Theorem~\ref{theorem:icbound}, namely,
\[
g(\Psi_n) = \max \{\;p_1-1,\; 2n\;\frac{1}{p_1}-\deg(\Psi_n)\;\}\]
{\em except for only two\/} values of $(p_2,p_3)$. In other words, it seems that the lower bound in Theorem~\ref{theorem:icbound} is almost always exactly the maximum gap.  The more detailed computational results (not given in the table) also suggest that the maximum gap is very close to the lower bound when it is not the same as the lower bound. Hence there is a hope for improving the upper bound. We leave it as an open problem. Any progress will require  full understanding on the intricate cancellations   occurring while adding and multiplying  polynomials.
 
\end{itemize}
\end{remark}

\section{Proof}\label{sec:proof}
In this section, we prove the three theorems given in the previous section.
We begin by  listing several short-hand notations that will be used throughout
the proofs without explicit references.
\begin{notation}[Notations used in the proof]
\label{notation:proof} \mbox{}
\begin{eqnarray*}
\varphi(n)       &=& \deg(\Phi_n)\\
\psi(n)       &=& \deg(\Psi_n)\\
\tdeg(f)   & = & \mbox{{\rm the trailing degree of a univariate polynomial \ensuremath{f}}}
\end{eqnarray*}
\end{notation}

\subsection{Proof of Theorem~\ref{theorem:c2p}}

Theorem~\ref{theorem:c2p} follows immediately from Lemma~\ref{lemma:upper} and Lemma~\ref{lemma:lower}.

\begin{lemma} \label{lemma:gap_add}
Let  $A$ and $B$ be polynomials. If there is no cancellation of terms while adding the two polynomials, then $$g(A+B) \;\;\leq\;\; \max\{\; g(A),\;g(B),\;\tdeg(B)-\deg(A),\; \tdeg(A)-\deg(B)\;\}$$
\end{lemma}
\begin{proof}
We consider several cases.
\begin{clists}
\citem{Case 1:} $\tdeg(B) > \deg(A)$. The gaps of $A+B$ occurs in $A, B$ and between $A$ and $B$. Thus
\[ g(A+B)  \;\;=\;\; \max\{\; g(A),\;g(B),\;\tdeg(B)-\deg(A)\;\}\]
Since $\tdeg(A)-\deg(B) < 0$, we have
\[  g(A+B) \;\;=\;\; \max\{\; g(A),\;g(B),\;\tdeg(B)-\deg(A),\; \tdeg(A)-\deg(B)\;\}\]

\citem{Case 2:} $\tdeg(A) > \deg(B)$. By switching the role of $A$ and $B$ in Case 1, 
we have
\[  g(A+B) \;\;=\;\; \max\{\; g(A),\;g(B),\;\tdeg(B)-\deg(A),\; \tdeg(A)-\deg(B)\;\}\]

\citem{Case 3:} $\deg(A) \geq \tdeg(B)$ and $\deg(B) \geq \tdeg(A)$. Since there is no cancellation of terms,  we have\[ g(A+B) \;\;\leq\;\; \max\{\; g(A),\;g(B) \;\}\]
Since $\tdeg(B)-\deg(A)\leq 0$ and $\tdeg(A)-\deg(B)\leq 0$, we have
\[g(A+B) \;\;\leq\;\; \max\{\; g(A),\;g(B),\;\tdeg(B)-\deg(A),\; \tdeg(A)-\deg(B)\;\}\]
\end{clists}
\end{proof}

\begin{lemma}\label{lemma:gap_mul}
Let $A$ and $B$ be polynomials.
If all the non-zero coefficients of $A$ have the same sign and all the non-zero coefficients of  $B$ have the same sign, then we have
$$g(AB) \;\;\leq\;\; \min\{u,v\}$$
where
$$ u = \max\{\; g(B), \;g(A)+\tdeg(B)-\deg(B)\;\}$$
$$ v = \max\{\; g(A), \;g(B)+\tdeg(A)-\deg(A)\;\} $$
\end{lemma}

\begin{proof}
 Let $A=\sum_{i=1}^t a_i x^{e_i}$ where $a_i > 0$ and $e_1 < e_2 < \cdots < e_t$. Let
 \[
 C_j  = \sum_{i=1}^j a_i x^{e_i} B
 \]
Note $AB = C_t$.

We claim that $g(C_j) \leq \max\{\; g(B), \;g(A)+\tdeg(B)-\deg(B)\}$ for $j=1,\ldots,t$. We will prove the claim by induction on $j$.
First, the claim is true for $j=1$ since
\[
  g(C_1) \;\;=\;\; g(a_1x^{e_1}B) \;\;=\;\;g(B)\;\;\leq\;\;\max\{\; g(B), \;g(A)+\tdeg(B)-\deg(B)\}
\]
Next assume that the claim is true for $j$. We will show that the claim is true for $j+1$. For this, note
that\[
g(C_{j+1}) = g(C_j + a_{j+1}x^{e_{j+1}}B)
\]
Since all the non-zero coefficients of $A$ have the same sign and all the non-zero coefficients of  $B$ have the same sign, there is no cancellation of terms in the above summation of $C_j$ and $a_{j+1}x^{e_{j+1}}B$.  Thus, from Lemma~\ref{lemma:gap_add}, we have
\[
g(C_{j+1}) \;\;\leq\;\; \max\{g(C_j),\;\;g(a_{j+1}x^{e_{j+1}}B),\;\tdeg(a_{j+1}x^{e_{j+1}}B)-\deg(C_j),\; \tdeg(C_j)-\deg(a_{j+1}x^{e_{j+1}}B)\;\}
\]

Note
\begin{eqnarray*}
g(C_j) &\leq& \max\{\; g(B), \;g(A)+\tdeg(B)-\deg(B)\;\} \\
g(a_{j+1}x^{e_{j+1}}B) &=& g(B) \\
\deg(C_j) &=& e_j +\deg(B) \\
\tdeg(C_j) &=& \tdeg(A) +\tdeg(B) \\
\deg(a_{j+1}x^{e_{j+1}}B) &=& e_{j+1} +\deg(B) \\
\tdeg(a_{j+1}x^{e_{j+1}}B) &=& e_{j+1} + \tdeg(B)
\end{eqnarray*}
Note
\begin{eqnarray*}
\tdeg(a_{j+1}x^{e_{j+1}}B)-\deg(C_j) &=& \left( e_{j+1} + \tdeg(B) \right) - \left(e_j +\deg(B)\right) \\
                                     &\leq&  g(A)+\tdeg(B)-\deg(B)   \\
                                     \\
\tdeg(C_j)-\deg(a_{j+1}x^{e_{j+1}}B) &=& \left(\tdeg(A) +\tdeg(B)\right) - \left(e_{j+1} +\deg(B)\right) \\
&\leq & 0
\end{eqnarray*}
Thus
\begin{eqnarray*}
g(C_{j+1}) & \leq & \max\left\{\max\{\; g(B), \;g(A)+\tdeg(B)-\deg(B)\;\},\;\;g(B),\;g(A)+\tdeg(B)-\deg(B)\;\right\} \\
 & = & \max\{g(B),\;\;g(A)+\tdeg(B)-\deg(B)\;\}
\end{eqnarray*}
Hence, we have proved the claim for $C_1,\ldots,C_t$. Since $AB=C_t$, we have
\[
g(AB) \;\leq\; u=\max\{g(B),\;\;g(A)+\tdeg(B)-\deg(B)\;\}
\]
By switching the role of $A$ and $B,$ we can also prove, in the identical way, that
\[
g(AB) \;\leq\; v=\max\{g(A),\;\;g(B)+\tdeg(A)-\deg(A)\;\}
\]
Hence we have $g(AB) \leq \min\{u,v\}$.
\end{proof}

\begin{lemma}\label{lemma:upper}
Let $p_1 < p_2$ be odd primes. Then we have
\[
g(\Phi_{p_1p_2})  \;\leq\; p_1-1
\]
\end{lemma}
\begin{proof}
From \cite{Lenstra1978,LamLeung1996, Thangadurai1999,Mor09},
$\Phi_{p_1p_2}$ has the form
\[
\Phi_{p_1p_2}(x)=\sum_{i=0}^{\rho}x^{ip_1}\;\cdot\;\sum_{j=0}^{\sigma}x^{jp_2}\;\;-\;\; x\cdot\sum_{i=0}^{p_2-2-\rho}x^{ip_1}\;\cdot\;\sum_{j=0}^{p_1-2-\sigma}x^{jp_2}\;\]
where $\rho$ and $\sigma$ are the unique integers such that $p_1p_2+1=(\rho+1)p_1+(\sigma+1)p_2$ with $0\leq\rho\leq p_2-2$ and $0\leq\sigma\leq p_1-2$.
It is also known that  accumulation/cancellation of terms does not occur when we expand the above expression for $\Phi_{p_1p_2}(x)$.
It will be more convenient to rewrite the above expression into the following equivalent form
\[
  \Phi_{p_1p_2}(x) =  A \cdot B + C \cdot D
\]
where
\begin{align*}
  A &=   \sum_{i=0}^{\rho}x^{ip_1}        &    B & =  \sum_{j=0}^{\sigma}x^{jp_2}\\
  C &=   \sum_{i=0}^{p_2-2-\rho}x^{ip_1}  &    D & =-x  \sum_{j=0}^{p_1-2-\sigma}x^{jp_2}
\end{align*}
Note that
\[
\begin{array}{lllllllll}
\tdeg(A) &=& 0   &\; &  \deg(A) &=& \rho p_1            &\; &  g(A)=p_1 \\
\tdeg(B) &=& 0   &\; &  \deg(B) &=& \sigma p_2          &\; &  g(B)=p_2  \\
\tdeg(C) &=& 0   &\; &  \deg(C) &=& (p_2-2-\rho)p_1    &\; & g(C)=p_1 \\
\tdeg(D) &=& 1   &\; &  \deg(D) &=& (p_1-2-\sigma)p_2+1  &\; & g(D)=p_2  \\
\end{array}
\]
Thus
\[
\begin{array}{lll}
g(B) + \tdeg(A)- \deg(A)  &=& p_2-\rho p_1 \\
                                 &=& p_2-p_1p_2-1+p_1+(\sigma+1)p_2  \\
                                 &\leq& p_2-p_1p_2-1+p_1+(p_1-1)p_2 \\
                                 &=&p_1-1\\
g(A) + \tdeg(B)- \deg (B)  &=& p_1-\sigma p_2 \\
                                          &\leq&  p_1 \\
g(D) + \tdeg(C)- \deg (C)  &=& p_2-(p_2-2-\rho)p_1 \\
                                 &=& p_2-p_2p_1+2p_1+\rho p_1 \\
                                 &=& p_2-p_2p_1+2p_1+p_1p_2+1-p_1-(\sigma+1)p_2 \\
                                 &=& p_1+1-\sigma p_2  \\
                                 &\leq & p_1+1\\
g(C) + \tdeg(D)- \deg (D)  &=& p_1+1-\left((p_1-2-\sigma)p_2+1\right)\\
                                 &=& p_1-(p_1-2-\sigma)p_2 \\
                                          &\leq&  p_1 \\
\end{array}
\]
By Lemma~\ref{lemma:gap_mul},
we have
\begin{eqnarray}
g(AB) &\leq& \min\{ \max\{p_2,p_1\},\max\{p_1,p_1-1\} \} = \min\{ p_2,p_1 \} = p_1     \label{eq:gap_AB}\\
g(CD)&\leq& \min\{ \max\{p_2,p_1\},\max\{p_1,p_1+1\} \} = \min\{ p_2,p_1 +1\} = p_1+1  \label{eq:gap_CD}
\end{eqnarray}

Here we could  apply Lemma~\ref{lemma:gap_add} to bound $g(AB+CD)$. However, it would  not be helpful since we would get a bound which is at least $p_1+1$. We want  a tighter bound, namely $p_1-1$.
 For this, we exploit the {\em  particular\/} way $AB$ and $CD$ are overlapping.
 We begin by noting
\[
\begin{array}{llllllll}
\tdeg(AB)  &=& 0   & \;\;\;\;\;\; &
\deg(AB)   &=& \rho p_1+\sigma p_2 = \varphi(p_1p_2)\\
\tdeg(CD) &=& 1   & \;\;\;\;\;\; &
\deg(CD)  &=& (p_2-2-\rho)p_1+(p_1-2-\sigma)p_2+1,
\end{array}
\]
Hence
\[
\begin{array}{ccccl}
\tdeg(CD) &-& \tdeg(AB)  &=& 1 \\
\deg(AB)  &-& \deg(CD)   &=&  2(\rho p_1+\sigma p_2-p_1p_2+p_1+p_2)-1\\
                 &  &                    &=& 2(1-p_1-p_2+p_1+p_2)-1=1
\end{array}
\]
So we have the following overlapping between $AB$ and $CD$ and the resulting $AB+CD$:
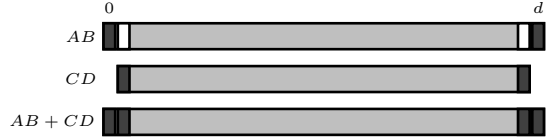
\begin{center}
\tiny
\psset{unit=0.19}
\begin{pspicture*}(-7,-7)(31,4)
  \psframe[linewidth=1pt,fillstyle=solid,fillcolor=lightgray](-0.5, 0)(30.5, 2)
  \psframe[linewidth=1pt,fillstyle=solid,fillcolor=darkgray] (-0.5, 0)( 0.5, 2)
  \psframe[linewidth=1pt,fillstyle=solid,fillcolor=white]    ( 0.5, 0)( 1.5, 2)
  \psframe[linewidth=1pt,fillstyle=solid,fillcolor=white]    (28.5, 0)(29.5, 2)
  \psframe[linewidth=1pt,fillstyle=solid,fillcolor=darkgray] (29.5, 0)(30.5, 2)
  \psframe[linewidth=1pt,fillstyle=solid,fillcolor=lightgray]( 0.5,-3)(29.5,-1)
  \psframe[linewidth=1pt,fillstyle=solid,fillcolor=darkgray] ( 0.5,-3)( 1.5,-1)
  \psframe[linewidth=1pt,fillstyle=solid,fillcolor=darkgray] (28.5,-3)(29.5,-1)
  \psframe[linewidth=1pt,fillstyle=solid,fillcolor=lightgray](-0.5,-6)(30.5,-4)
  \psframe[linewidth=1pt,fillstyle=solid,fillcolor=darkgray] (-0.5,-6)( 0.5,-4)
  \psframe[linewidth=1pt,fillstyle=solid,fillcolor=darkgray] ( 0.5,-6)( 1.5,-4)
  \psframe[linewidth=1pt,fillstyle=solid,fillcolor=darkgray] (28.5,-6)(29.5,-4)
  \psframe[linewidth=1pt,fillstyle=solid,fillcolor=darkgray] (29.5,-6)(30.5,-4)
  \rput(-2, 1){$AB$}
  \rput(-2,-2){$CD$} 
  \rput(-4,-5){$AB+CD$} 
  \rput(0,3){$0$}
  \rput(30,3){$d$}
\end{pspicture*}
\end{center}    
where each exponent is colored in black, white and gray to indicate that
the exponent occurs, does not occur, and may or may not occur, respectively.
The letter $d$ is the shorthand for the degree of the polynomial~$AB$.
Note that the exponents $1$ and $d-1$ in $AB$ are colored in white because
$AB$ and $CD$ do not share any exponents.
As the result, the exponents $0,1,d-1,d$ occur in $AB+CD$, and are colored in black. 

Due to the way the polynomials $AB$ and $CD$ are overlapped, while adding $CD$ to $AB$,  none of  the terms of $
CD$ can ever increase the gaps already in $AB$.
Hence $$g(AB+CD) \;\leq\;\; g(AB)$$
Thus from Formula~(\ref{eq:gap_AB}) we have
\[
    g(\Phi_{p_1p_2}) \;\;=\;\; g(AB+CD) \;\leq \;p_1
\]

Hence in order to prove the first claim: $g(\Phi_{p_1p_2}) \leq p_1-1$, 
it only remains to show that $g(\Phi_{p_1p_2})\neq p_1$.
We will do so by contradiction. 
Suppose that $g(\Phi_{p_1p_2}) = p_1$. 
Then there must occur two exponents, 
say $\alpha$ and $\beta$, in the polynomial~$AB+CD$ such that $\beta-\alpha = p_1$
and all the exponents in between them do not occur in $AB+CD$. 
Note that $\alpha\geq$1 and $\beta\leq d-1$. Then we are  in the   situation
described by the colorings in the following diagram\begin{center}
\tiny
\psset{unit=0.19}
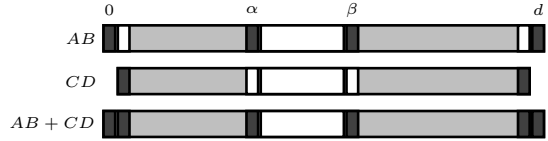
\begin{pspicture*}(-7,-7)(31,4)
  \psframe[linewidth=1pt,fillstyle=solid,fillcolor=lightgray](-0.5, 0)(30.5, 2)
  \psframe[linewidth=1pt,fillstyle=solid,fillcolor=darkgray] (-0.5, 0)( 0.5, 2)
  \psframe[linewidth=1pt,fillstyle=solid,fillcolor=white]    ( 0.5, 0)( 1.5, 2)
  \psframe[linewidth=1pt,fillstyle=solid,fillcolor=white]    (28.5, 0)(29.5, 2)
  \psframe[linewidth=1pt,fillstyle=solid,fillcolor=darkgray] (29.5, 0)(30.5, 2)
  \psframe[linewidth=1pt,fillstyle=solid,fillcolor=lightgray]( 0.5,-3)(29.5,-1)
  \psframe[linewidth=1pt,fillstyle=solid,fillcolor=darkgray] ( 0.5,-3)( 1.5,-1)
  \psframe[linewidth=1pt,fillstyle=solid,fillcolor=darkgray] (28.5,-3)(29.5,-1)
  \psframe[linewidth=1pt,fillstyle=solid,fillcolor=lightgray](-0.5,-6)(30.5,-4)
  \psframe[linewidth=1pt,fillstyle=solid,fillcolor=darkgray] (-0.5,-6)( 0.5,-4)
  \psframe[linewidth=1pt,fillstyle=solid,fillcolor=darkgray] ( 0.5,-6)( 1.5,-4)
  \psframe[linewidth=1pt,fillstyle=solid,fillcolor=darkgray] (28.5,-6)(29.5,-4)
  \psframe[linewidth=1pt,fillstyle=solid,fillcolor=darkgray] (29.5,-6)(30.5,-4)
  \rput(-2, 1){$AB$}
  \rput(-2,-2){$CD$} 
  \rput(-4,-5){$AB+CD$} 
  \rput(0,3){$0$}
  \rput(30,3){$d$}

  \psframe[linewidth=1pt,fillstyle=solid,fillcolor=darkgray]( 9.5, 0)(10.5, 2)
  \psframe[linewidth=1pt,fillstyle=solid,fillcolor=white]   (10.5, 0)(16.5, 2)
  \psframe[linewidth=1pt,fillstyle=solid,fillcolor=darkgray](16.5, 0)(17.5, 2)
  \psframe[linewidth=1pt,fillstyle=solid,fillcolor=white]   ( 9.5,-3)(10.5,-1)
  \psframe[linewidth=1pt,fillstyle=solid,fillcolor=white]   (10.5,-3)(16.5,-1)
  \psframe[linewidth=1pt,fillstyle=solid,fillcolor=white]   (16.5,-3)(17.5,-1)
  \psframe[linewidth=1pt,fillstyle=solid,fillcolor=darkgray]( 9.5,-6)(10.5,-4)
  \psframe[linewidth=1pt,fillstyle=solid,fillcolor=darkgray](16.5,-6)(17.5,-4)
  \psframe[linewidth=1pt,fillstyle=solid,fillcolor=white]   (10.5,-6)(16.5,-4)
  \rput(10,3){$\alpha$}
  \rput(17,3){$\beta$} 
\end{pspicture*}
\end{center}
In the above diagram, the exponents $\alpha$ and $\beta$ in the polynomial $AB+CD$ are colored in black
because they occur in  $AB+CD$ and all the exponents in between them are colored in white
because they do not occur in $AB+CD$.
Since there is no cancellation of terms while summing $AB$ and $CD$,
all the exponents in between $\alpha$ and $\beta$
in $AB$ and $CD$ cannot occur either, hence colored in white also.
Now from  Formula~(\ref{eq:gap_AB}), we have  $g(AB)\leq p_1$. Since $\beta-\alpha = p_1$, 
the exponents $\alpha$ and $\beta$ must occur in $AB$, hence colored in black.
Since $AB$ and $CD$ do not share any exponents,
the exponents $\alpha$ and $\beta$ must {\em not\/} occur in $CD$, hence colored in white.
Thus we have justified all the colorings in the above diagram.  

Now we are ready to derive a contradiction. 
From the diagram, we see that 
\[
g(CD) \geq (\beta+1)-(\alpha-1) = \beta-\alpha+2 = p_1+2
\]
But from Formula~(\ref{eq:gap_CD}), we have 
\[
 g(CD) \leq p_1 + 1
\]          
This is a contradiction.
Hence $g(\Phi_{p_1p_2})\neq p_1$. Thus we finally have   
\[
g(\Phi_{p_1p_2}) \;\leq\; p_1-1
\]
\end{proof} 
 
\begin{lemma}\label{lemma:lower}
Let $p_1 < p_2$ be odd primes. Then we have
\[
g(\Phi_{p_1p_2}) \;\geq\; p_1-1
\]
\end{lemma}
\begin{proof}
We will show this by finding a gap of size $p_1-1$. We begin by recalling
\begin{eqnarray*}
AB &=& \sum_{i=0}^{\rho}x^{ip_1}\;\cdot\;\sum_{j=0}^{\sigma}x^{jp_2}  \\
CD &=& -\;\; x\cdot\sum_{i=0}^{p_2-2-\rho}x^{ip_1}\;\cdot\;\sum_{j=0}^{p_1-2-\sigma}x^{jp_2}
\end{eqnarray*}
where $\rho$ and $\sigma$ are the unique integers such that 
\[
p_1p_2+1=(\rho+1)p_1+(\sigma+1)p_2
\] 
with $0\leq\rho\leq p_2-2$ and $0\leq\sigma\leq p_1-2$.

We claim that $\rho \geq 1$. Suppose otherwise. Then $\rho=0$ and thus we have $$p_1p_2+1=p_1+(\sigma+1)p_2$$ 
Taking both sides modulo $p_2$, we see  $1 \equiv p_1 \pmod{p_2}$.  
This contradicts the fact $1< p_1 < p_2$.
Hence $\rho \geq 1$. 

Thus the polynomial $AB$ must have the following form: 
\[
AB = 1 + x^{p_1} +  \mbox{terms of degree higher than $p_1$ if there is any }
\]
On the other hand, the polynomial $CD$ must have the following form:
\[
CD = -x -  \mbox{terms of degree higher than $p_1$ if there is any}
\]
Thus the polynomial $AB + CD$ must have the following form:
\begin{equation}\label{eq:phi_p1p2_initial_gap}
\Phi_{p_1p_2}(x)=AB+CD = 1 -x+x^{p_1} +  \mbox{terms of degree higher than $p_1$}
\end{equation}
Thus there is a gap of size $p_1-1$ between $x$ and $x^{p_1}$.  Hence we finally have
\[
g(\Phi_{p_1p_2}) \;\geq\; p_1-1
\]
\end{proof}

\subsection{Proof of Theorem~\ref{theorem:ic3p}}

Theorem~\ref{theorem:ic3p} follows immediately from Lemma~\ref{lemma:D1_op_op_op_exact}, Lemma~\ref{lemma:D2_op_op_op_exact} and Lemma~\ref{lemma:D1D2_C}.

\begin{lemma}\label{lemma:D1_op_op_op_exact}
Let $n=p_1p_2p_3$ where $p_1<p_2<p_3$ are odd primes satisfying
\[
\begin{array}{lrcl}
\s{D1}:&  2n\;\frac{1}{p_1}&>&\frac{4}{3}\deg(\Psi_n)
\end{array}
\]
Then we have
\[g(\Psi_n)\;=\;2n\;\frac{1}{p_1}-\deg(\Psi_n)\]
\end{lemma}

\begin{proof}
By Lemma 2 in~\cite{Mor09} we have
\begin{eqnarray*}
\Psi_{p_1p_2p_3}(x) &=& \Phi_{p_1p_2}(x)\cdot\Psi_{p_1p_2}(x^{p_3}) \\
              &=& \Phi_{p_1p_2}(x)\cdot\Phi_{p_1}(x^{p_3})\cdot\Psi_{p_1}(x^{p_2p_3}) \\
              &=& \Phi_{p_1p_2}(x)\cdot\Phi_{p_1}(x^{p_3})\cdot(-1+x^{p_2p_3})   
\end{eqnarray*}
We expand the above expression and name the parts as follows.
\[
\underbrack{-\underbrack{\Phi_{p_1p_2}(x)\cdot\Phi_{p_1}(x^{p_3})         }_{A_0}
           +\underbrack{x^{p_2p_3}\Phi_{p_1p_2}(x)\cdot\Phi_{p_1}(x^{p_3})}_{A_1}}_{\Psi_{p_1p_2p_3}(x)}
\]
Let $\lambda$ be the gap, if exists, between $A_0$ and $A_1$, that is,  $\tdeg(A_1) -\deg(A_0)$. 
Note
\begin{eqnarray*}
\deg(A_0)  &=& \deg(\Psi_n)- p_2p_3 = \deg(\Psi_n)- n\;\frac{1}{p_1} \\
\tdeg(A_1) &=& p_2p_3  = n\;\frac{1}{p_1}   
\end{eqnarray*}
Thus
\[
\lambda \;=\;  n\;\frac{1}{p_1}- (\deg(\Psi_n)- n\;\frac{1}{p_1}) 
        \;=\;  2n\frac{1}{p_1} - \deg(\Psi_n)
\]
Note that
\begin{eqnarray*}
\lambda  &=& 2n\frac{1}{p_1} - \deg(\Psi_n) \\
                    &=& 3n\frac{1}{p_1} - 2\deg(\Psi_n)+\deg(A_0) \\
                    &=& \frac{3}{2}\left(2n\frac{1}{p_1} - \frac{4}{3}\deg(\Psi_n)\right)+\deg(A_0) \\
                    &>& \deg(A_0) \\
                    &\geq& g(A_0)=g(A_1)
\end{eqnarray*}
Thus $\lambda >0$ and the gap between $A_0$ and $A_1$ exists. 
Hence \[
g(\Psi_n) =\max\{\;g(A_0),\;\lambda, \;g(A_1)\;\}
=\lambda=\;2n\;\frac{1}{p_1}-\deg(\Psi_n)
\]

\end{proof}
\begin{lemma}\label{lemma:D2_op_op_op_exact}
Let $n=p_1p_2p_3$ where $p_1<p_2<p_3$ are odd primes satisfying:
\begin{eqnarray*}
    \s{D2}: & 2p_3\;\;& > \;\;\; p_2(p_1-1)
\end{eqnarray*}
Then we have
\[g(\Psi_n)\;=2n\;\frac{1}{p_1} -\deg(\Psi_n)\]
\end{lemma}
\begin{proof}  
By Lemma 2 in~\cite{Mor09} we have
\begin{eqnarray*}
\Psi_{p_1p_2p_3}(x) &=& \Phi_{p_1p_2}(x)\cdot\Psi_{p_1p_2}(x^{p_3}) \\
              &=& \Phi_{p_1p_2}(x)\cdot\Phi_{p_1}(x^{p_3})\cdot\Psi_{p_1}(x^{p_2p_3}) \\
              &=& \Phi_{p_1p_2}(x)\cdot(1+x^{p_3}+\ldots+x^{(p_1-1)p_3})\cdot(-1+x^{p_2p_3})   
\end{eqnarray*}
We expand the above expression and name the parts as follows.
\[
\underbrack{-\underbrack{         (\underbrack{\Phi_{p_1p_2}(x)}_{B_0}+\ldots+\underbrack{x^{(p_1-1)p_3}\Phi_{p_1p_2}(x)}_{B_{p_1-1}})}_{A_0}
           +\underbrack{x^{p_2p_3}(\underbrack{\Phi_{p_1p_2}(x)}_{B_0}+\ldots+\underbrack{x^{(p_1-1)p_3}\Phi_{p_1p_2}(x)}_{B_{p_1-1}})}_{A_1}}_{\Psi_{p_1p_2p_3}(x)}
\]
Let 
\[
\lambda = \tdeg(A_1) - \deg(A_0) = 2n\frac{1}{p_1}-\psi(n)
\]
From $\s{D2},$ we have 
\begin{eqnarray}
\lambda&=&2n\;\frac{1}{p_1}-\psi(n) \nonumber\\
       &=&2p_2p_3-p_1p_2p_3+(p_1-1)(p_2-1)(p_3-1) \nonumber\\
       &=& p_2p_3-p_1p_3+p_3+p_1-1-p_2(p_1-1) \nonumber\\
       &>& (p_2-p_1)p_3 + p_3+p_1-1 - 2p_3  \nonumber\\
       &=& (p_2-p_1-2)p_3 + p_3 + p_1-1 \label{eq:lambda}
\end{eqnarray}
Thus $\lambda > 0$, i.e. there is no overlap between $A_0$ and $A_1.$ Note that $g(A_0)= g(A_1)$. Thus 
\[
g(\Psi_n)=\max\{\lambda, \;\; g(A_0)\}
\]
We claim that 
$\lambda \;>\; g(A_0).$
Note 
\[
   \s{D2} \;\; \Longleftrightarrow \;\; \varphi(p_1p_2)  < 2p_3-(p_1-1)
\]
We will split the proof into the  following two cases:
\begin{clists}
\citem{Case 1:} $\varphi(p_1p_2) < p_3$.

Note that $\deg(B_0) = \varphi(p_1p_2)$ and $\tdeg(B_1) = p_3$. Hence there is no overlap 
between $B_0$ and~$B_1$.  Likewise there is no overlap between $B_i$ and~$B_{i+1}$ for all $i=1,\ldots, p_1-2$.
Note
\[
g(B_0)\;=\;g(B_1)\;=\;\ldots\;=\;g(B_{p_1-1})\;=\;g(\Phi_{p_1p_2})\;=\;p_1-1
\]
from Theorem~\ref{theorem:c2p}.
Hence 
\[
g(A_0) = \max\{\; p_3 -\varphi(p_1p_2), \;p_1-1\;\}
\]
 From  Eq.~(\ref{eq:lambda}), we have
\begin{align*}
\lambda & > (p_2-p_1-2)p_3 + p_3 + p_1-1 \;>\;p_3\;>\;p_3-\varphi(p_1p_2)\\
\lambda & > (p_2-p_1-2)p_{3} + p_3 + p_1-1 \;>\; p_1 - 1
\end{align*}
Thus we have proved that $\lambda \;>\; g(A_0)$ when $p_3>\varphi(p_1p_2)$.

\citem{Case 2:} $p_3\leq\varphi(p_1p_2)< 2p_3-(p_1-1)$.\\
Note 
\[
\tdeg(B_2) - \deg(B_0) \;=\; 2p_3 - \varphi(p_1p_2) \; >\; 0
\]
Thus $B_0, B_1, \ldots, B_{p_1-1}$ overlap as the following diagram shows.
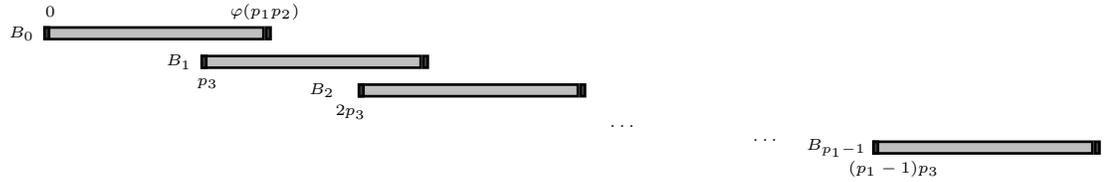
\begin{center}
\tiny
\psset{unit=0.19}
\begin{pspicture*}(-3,-9)(75,4)
  \def\poly{
    \psframe[linewidth=1pt,fillstyle=solid,fillcolor=lightgray](-0.5,1)(15,2)
    \psframe[linewidth=1pt,fillstyle=solid,fillcolor=darkgray](-0.5,1)( 0,2)
    \psframe[linewidth=1pt,fillstyle=solid,fillcolor=darkgray](15,1)(15.5,2)
  }
  \multips(0,0)(11,-2){3}{\poly}
  \multips(58,-8)(10,-2){1}{\poly}  
  \rput(-2,1.5){$B_0$}
  \rput(9,-0.5){$B_1$}
  \rput(19,-2.5){$B_2$}
  \rput(55,-6.5){$B_{p_1-1}$}
  \rput( 0,3){$0$}
  \rput(15,3){$\varphi(p_1p_2)$}
  \rput(11,-2){$p_3$}
  \rput(21,-4){$2p_3$}
  \rput(59,-8){$(p_1-1)p_3$}
  \rput(40,-5){\textbf{$\ldots$}}
  \rput(50,-6){\textbf{$\ldots$}}
\end{pspicture*}
\end{center}
In the above diagram, the tail exponent and the leading exponent of $B_0$
are colored in black to indicate that they actually occur in $B_0$. 
The other exponents are colored in gray to indicate that they may or may not
occur. The same is done for  $B_2,\ldots,B_{p_1-1}$ since they have the same sparsity structure (shifting does not change the sparsity structure).
In $B_0$, there occurs at least one exponent between 0 and $p_3$. Otherwise we would have 
$p_3 - 0 > p_1-1=g(B_0)$ which is impossible.
Let $\alpha$ be the largest such exponent.
Then $p_3-\alpha\;\leq \;p_1-1$. Since
\[
    2p_{3} \; > \; \alpha+p_3 \; \geq \; 2p_3-(p_1-1) \; > \; \varphi(p_1p_2)
\]
the exponent $\alpha+p_3$ lies between $\varphi(p_1p_2)$ and $2p_3$ in $B_1.$

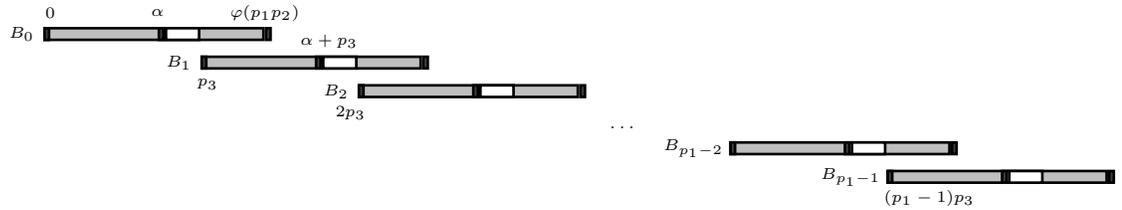
\begin{center}
\tiny
\psset{unit=0.19}
\begin{pspicture*}(-3,-11)(75,4)
  \def\poly{
    \psframe[linewidth=1pt,fillstyle=solid,fillcolor=lightgray](-0.5,1)(15,2)
    \psframe[linewidth=1pt,fillstyle=solid,fillcolor=darkgray](-0.5,1)( 0,2)
    \psframe[linewidth=1pt,fillstyle=solid,fillcolor=darkgray](15,1)(15.5,2)
    \psframe[linewidth=1pt,fillstyle=solid,fillcolor=darkgray](7.5,1)(8,2)
    \psframe[linewidth=1pt,fillstyle=solid,fillcolor=white](8,1)(10.5,2)
  }
  \multips(0,0)(11,-2){3}{\poly}
  \multips(48,-8)(11,-2){2}{\poly}  
  \rput(-2,1.5){$B_0$}
  \rput(9,-0.5){$B_1$}
  \rput(20,-2.5){$B_2$}
  \rput(45,-6.5){$B_{p_1-2}$}
  \rput(56,-8.5){$B_{p_1-1}$}
  \rput( 0,3){$0$}
  \rput(15,3){$\varphi(p_1p_2)$}
  \rput(11,-2){$p_3$}
  \rput(21,-4){$2p_3$}
  \rput(61.5,-10){$(p_1-1)p_3$}
    \rput(7.5,3){$\alpha$}
        \rput(19.5,1){$\alpha+p_3$}
  \rput(40,-5){\textbf{$\ldots$}}
\end{pspicture*}
\end{center}
Now we consider the polynomials $E_1, L_1, L_2, \ldots, L_{p_1-2}$ and $E_2$ indicated in the following diagram
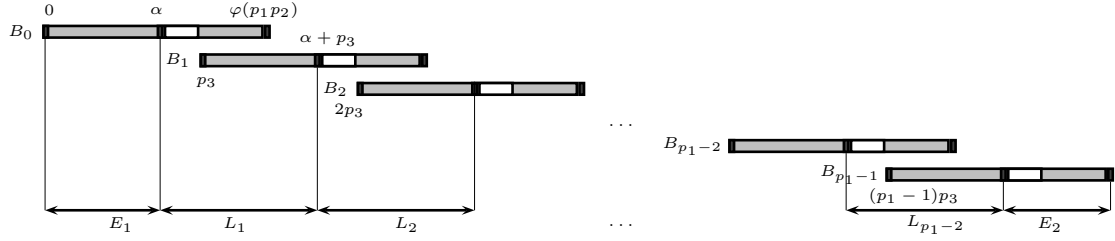
\begin{center}
\tiny
\psset{unit=0.19}
\begin{pspicture*}(-3,-15)(75,4)
  \def\poly{
    \psframe[linewidth=1pt,fillstyle=solid,fillcolor=lightgray](-0.5,1)(15,2)
    \psframe[linewidth=1pt,fillstyle=solid,fillcolor=darkgray](-0.5,1)( 0,2)
    \psframe[linewidth=1pt,fillstyle=solid,fillcolor=darkgray](15,1)(15.5,2)
    \psframe[linewidth=1pt,fillstyle=solid,fillcolor=darkgray](7.5,1)(8,2)
    \psframe[linewidth=1pt,fillstyle=solid,fillcolor=white](8,1)(10.5,2)
  }
  \multips(0,0)(11,-2){3}{\poly}
  \multips(48,-8)(11,-2){2}{\poly}  
  \rput(-2,1.5){$B_0$}
  \rput(9,-0.5){$B_1$}
  \rput(20,-2.5){$B_2$}
  \rput(45,-6.5){$B_{p_1-2}$}
  \rput(56,-8.5){$B_{p_1-1}$}
  \rput( 0,3){$0$}
  \rput(15,3){$\varphi(p_1p_2)$}
  \rput(11,-2){$p_3$}
  \rput(21,-4){$2p_3$}
  \rput(60.5,-10){$(p_1-1)p_3$}
    \rput(7.5,3){$\alpha$}
        \rput(19.5,1){$\alpha+p_3$}
  \rput(40,-5){\textbf{$\ldots$}}
   
  \psline[linewidth=1pt, arrows=<->](-0.3,-11)(7.75,-11)
  \psline[linewidth=1pt, arrows=<->](7.75,-11)(18.75,-11)
  \psline[linewidth=1pt, arrows=<->](18.75,-11)(29.75,-11)
  \psline[linewidth=1pt, arrows=<->](55.75,-11)(66.75,-11)
  \psline[linewidth=1pt, arrows=<->](66.75,-11)(74.25,-11)

  \psline[linewidth=0pt,linestyle=solid](-0.3,-11)(-0.3,1)
  \psline[linewidth=0pt,linestyle=solid](7.75,-11)(7.75,1)
  \psline[linewidth=0pt,linestyle=solid](18.75,-11)(18.75,-1)
  \psline[linewidth=0pt,linestyle=solid](29.75,-11)(29.75,-2)
  \psline[linewidth=0pt,linestyle=solid](55.75,-7)(55.75,-11)
  \psline[linewidth=0pt,linestyle=solid](66.75,-9)(66.75,-11)
  \psline[linewidth=0pt,linestyle=solid](74.25,-9)(74.25,-11)
  \rput(5,-12){$E_1$}
  \rput(13,-12){$L_1$}
    \rput(25,-12){$L_2$}
  \rput(62,-12){$L_{p_1-2}$}
  \rput(70,-12){$E_{2}$}
  \rput(40,-12){$\ldots$}
\end{pspicture*}
\end{center}
where $\deg(E_1)=\alpha, \;\; \tdeg(E_2)=\psi(n)-\alpha\;$ and $\;\deg(L_i) = \tdeg(L_{i+1})$.
Since  $L_1,L_2,\;\ldots,\;L_{p_1-2}$ have the same gap structure, we have
\[  g(L_1)=g(L_2)=\cdots=g(L_{p_1-2})
\]
Hence, we have\[
g(A_0) =\max\{\;g(E_1),\;g(E_2),\;g(L_1)\;\}
\] 

From Theorem~\ref{theorem:c2p} and Eq.~(\ref{eq:lambda}), we have
\[
\lambda \;\;>\;\; (p_2-p_1-2)p_3 + p_3 + p_1-1 \;\;>\;\; p_1-1 \;\; = \;\;  g(\Phi_{p_1p_2}) \;\;\geq\;\; g(E_1),\; g(E_2)
\]
Note
\begin{eqnarray*}
g(L_1)     &\leq&  (\alpha+p_3) - \alpha = p_3 
\end{eqnarray*}
From Eq.~(\ref{eq:lambda}), we have
\begin{eqnarray*}
\lambda &=& 2n\;\frac{1}{p_1}-\psi(n)  \\
        &>&  p_3(p_2-p_1-2) + p_3 + p_1-1\\
        &>&  p_3 \\
        &\geq& g(L_1)
\end{eqnarray*}
Thus we have proved that $\lambda > g(A_0)=\max\{g(E_1), \; g(L_1)\}$ when $p_3 \leq \varphi(p_1p_2) < 2p_3 - (p_1-1)$.
\end{clists}
\end{proof}

\begin{lemma}\label{lemma:D1D2_C}
Let $n=p_1p_2p_3$ where $p_1<p_2<p_3$ are odd primes. Then we have
\[
    \s{C1}\;\; \vee \;\; \s{C2}\;\;\; \Longrightarrow \;\;\s{D1}\;\; \vee \;\; \s{D2}
\]
where 
\[
\begin{array}{lrcl}
\medskip      \s{C1}:&  4(p_1-1)      &\leq&p_2\\
\medskip      \s{C2}:&  p_1^2 &\leq&  p_3\\
\medskip      \s{D1}:&  2n\;\frac{1}{p_1}  &>& \frac{4}{3}\deg(\Psi_n) \\
\medskip      \s{D2}:& p_2(p_1-1)& < & 2p_3
\end{array}
\]
\end{lemma}
\begin{proof}
Let $n=p_1p_2p_3$ where $p_1<p_2<p_3$ are odd primes. We will prove the contrapositive.
\[
   \lnot\s{D1} \;\; \wedge \;\; \lnot\s{D2} \;\;\;\; \Longrightarrow \;\;\;\;    \lnot\s{C1} \;\; \wedge \;\; \lnot\s{C2}
\]
Let 
\[
V= \{\;(p_1,p_2,p_3)\;:\; \lnot \s{D1} \;\;\land\;\;\lnot \s{D2}\;\;\land \;\; p_1<p_2 \;\;\land\;\; p_2<p_3 \;\}
\]
It suffices to prove
\[
(p_1,p_2,p_3) \in V  \;\;\;\;\Longrightarrow\;\;\; \lnot\s{C1} \;\; \wedge \;\; \lnot\s{C2}
\]
Note 
\[
V= \{\;(p_1,p_2,p_3)\;:\;h_1 \leq 0 \;\;\land\;\;h_2 \leq 0\;\;\land \;\; h_3 <0 \;\;\land\;\; h_4 <0 \;\}
\] 
where 
\begin{eqnarray*}
  h_1 &=& 2n\;\frac{1}{p_1}  - \frac{4}{3}\deg(\Psi_n)\\
  h_2 &=& 2p_3-p_2(p_1-1)\\
  h_3 &=& p_1-p_2 \\
  h_4 &=& p_2-p_3
\end{eqnarray*}
The shaded area in the plot below shows the cross section of the set $V$ for a fixed $p_1$.
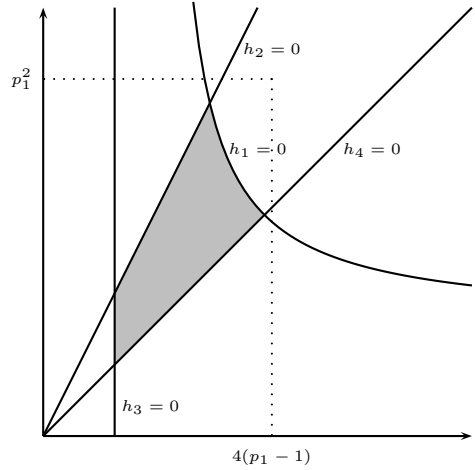
\begin{center}
\tiny
\psset{unit=0.19}
\def\p1{5}
\pspicture(-0.5,-1)(30,31)
\psaxes[ticks=none]{->}(30,30)
\psclip
{
 \pscustom[linestyle=none]
 {%
   \psline(0,0)(0,30.4)
   \lineto(10.5,30.4)
   \psplot{10.5}{30}{2 \p1 1 sub mul x 1 sub mul x 2 \p1 mul sub 2 add div}
   \lineto(30,0)
   \lineto(0,0)
 }
 \pscustom[linestyle=none]
 {%
   \psplot{0}{15}{x 2 div \p1 1 sub mul}
   \lineto(30.4,30.4)
   \lineto(30,0)
   \lineto(0,0)
 }
 \pscustom[linestyle=none]
 {%
   \psline(\p1,0)(\p1,30)
   \lineto(50,30)
   \lineto(50, 0)
 }
 \pscustom[linestyle=none]
 {%
   \psplot{0}{30}{x}
   \lineto(0,30)
   \lineto(0,0)
 }
}
\psframe*[linecolor=lightgray](30,50)
\endpsclip
\psplot{10.5}{30}{2 \p1 1 sub mul x 1 sub mul x 2 \p1 mul sub 2 add div}
\psplot{0}{15}{x 2 div \p1 1 sub mul}
\psline(\p1,0)(\p1,30)
\psplot{0}{30}{x}
\psline[linestyle=dotted](0,25)(16,25)
\psline[linestyle=dotted](16,0)(16,25)
\rput(15, 20){$h_1=0$}
\rput(16, 27){$h_2=0$}
\rput(7.5,2){$h_3=0$}
\rput(23,20){$h_4=0$}
\rput(-1.5,25){$p_1^2$}
\rput(16,-1.5){$4(p_1-1)$}
\endpspicture
\end{center}
By finding the $p_2$ coordinate of the intersection point between  the curves $h_1=0$ and $h_4=0$, we have  
\begin{eqnarray*}
(p_1,p_2,p_3)\in V   &\Longrightarrow&   p_2 \;\;\leq\;\;2(p_1-1)+\sqrt{4p_1^2-10p_1+6} \;\;  \\
                     &\Longrightarrow&   p_2 \;\;<\;\;2(p_1-1)+\sqrt{4(p_1-1)^2}  \\
                     &\Longrightarrow&   p_2 \;\;<\;\; 4(p_1-1) \\
                     &\Longrightarrow&   \lnot \s{C1}
\end{eqnarray*}
By finding the $p_3$ coordinate of the intersection point between  the curves $h_1=0$ and $h_2=0$, we have 
\begin{eqnarray*}
(p_1,p_2,p_3)\in V   &\Longrightarrow&   p_3 \;\;\leq\;\; \frac{1}{2}(p_1-1)\left(p_1+1+\sqrt{p_1^2+2p_1-3}\right) \\
                     &\Longrightarrow&   p_3 \;\;<\;\; \frac{1}{2}(p_1-1)\left(p_1+1+\sqrt{(p_1+1)^2}\right) \\
                     &\Longrightarrow&   p_3 \;\;<\;\; (p_1-1)(p_1+1)  \\
                     &\Longrightarrow&   p_3 \;\;<\;\; p_1^2 \\
                     &\Longrightarrow&   \lnot \s{C2}
\end{eqnarray*}
\end{proof}

\subsection{Proof of Theorem~\ref{theorem:icbound}}
Theorem~\ref{theorem:icbound} follows immediately from Lemma~\ref{lemma:low_bound_op_op_op} and Lemma~\ref{lemma:upper_bound_op_op_op}.
\begin{lemma}\label{lemma:low_bound_op_op_op}
Let $n=p_1p_2p_3$ where $p_1<p_2<p_3$ are odd primes. We have
\[    \max\{\;p_1-1, \;2n\;\frac{1}{p_1}-\deg(\Psi_n)\;\} \;\; \leq \;\; g(\Psi_n)
\]
\end{lemma}
\begin{proof}
We recall the diagram in the proof of Lemma~\ref{lemma:D2_op_op_op_exact}:
\[
\underbrack{-\underbrack{         (\underbrack{\Phi_{p_1p_2}(x)}_{B_0}+\ldots+\underbrack{x^{(p_1-1)p_3}\Phi_{p_1p_2}(x)}_{B_{p_1-1}})}_{A_0}
           +\underbrack{x^{p_2p_3}(\underbrack{\Phi_{p_1p_2}(x)}_{B_0}+\ldots+\underbrack{x^{(p_1-1)p_3}\Phi_{p_1p_2}(x)}_{B_{p_1-1}})}_{A_1}}_{\Psi_{p_1p_2p_3}(x)}
\]
Let  $\lambda=\tdeg(A_1)-\deg(A_0)$. Then we have
\[\lambda= p_2p_3-(\deg(\Psi_n)-p_2p_3)=2p_2p_3 -\deg(\Psi_n)=2n\;\frac{1}{p_1} -\deg(\Psi_n)\]
 If $\lambda \leq 0$, then $\lambda \leq g(\Psi_{n})$ obviously. If $\lambda > 0$, there exists a gap between $A_0$ and $A_1$, thus  $\lambda \leq g(\Psi_{n})$.
We recall Eq.~(\ref{eq:phi_p1p2_initial_gap}):
\begin{equation*}
\Phi_{p_1p_2}(x)= 1 -x+x^{p_1} +  \mbox{terms of degree higher than $p_1$}
\end{equation*}
Therefore there exists a gap in $B_0$ of size $p_1-1$. Since $p_1 < p_3$, we have
\begin{equation*}
\Psi_{n}(x) = 1-x+x^{p_1} +  \mbox{terms of degree higher than} \; p_1
\end{equation*}
Hence, $p_1-1 \leq g(\Psi_{n})$.
\end{proof}

\begin{lemma}\label{lemma:upper_bound_op_op_op}
Let $n=p_1p_2p_3$ where $p_1<p_2<p_3$ are odd primes. Then
\[
g(\Psi_n) \;\;<\;\; 2n\left(\frac{1}{p_1}+\frac{1}{p_2}+\frac{1}{p_3}\right)-\deg(\Psi_n)
\]
\end{lemma}

\begin{proof}
Let $ U= 2n(\frac{1}{p_1}+\frac{1}{p_2}+\frac{1}{p_3}) -\deg(\Psi_n)$. Then Lemma follows from the following Claims.
\begin{clists}
\citem{Claim 1:} $g(\Psi_n)                  \leq \max\{p_1-1, \deg(\Psi_n)-2(p_3-(p_1-1))\}.$\\
Let $\alpha$ be the largest exponent less than $p_3$  occurring in $\Psi_n$ and $\beta=\psi(n)-\alpha$.
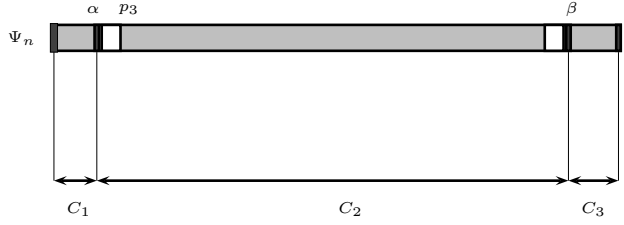
\begin{center}
\tiny
\psset{unit=0.19}
\begin{pspicture*}(-3,-13)(75,4)
  \def\poly{
    \psframe[linewidth=1pt,fillstyle=solid,fillcolor=lightgray](0,0)(40,2)
    \psframe[linewidth=0pt,fillstyle=solid,fillcolor=darkgray](0,0)(0.5,2)
    \psframe[linewidth=1pt,fillstyle=solid,fillcolor=darkgray](3,0)(3.5,2)
    \psframe[linewidth=1pt,fillstyle=solid,fillcolor=white](3.5,0)(5,2)
    \psframe[linewidth=1pt,fillstyle=solid,fillcolor=white](34.5,0)(36,2)
    \psframe[linewidth=1pt,fillstyle=solid,fillcolor=darkgray](36,0)(36.5,2)
    \psframe[linewidth=1pt,fillstyle=solid,fillcolor=darkgray](39.5,0)(40,2)
  }
  \multips(0,0)(40,-3){1}{\poly}
  \rput(-2,1){$\Psi_n$}

  \rput( 3,3){$\alpha$}
  \rput( 5.5,3){$p_3$}
  \rput( 36.5,3){$\beta$}

  \psline[linewidth=1pt, arrows=<->](0.25,-9)(3.25,-9)
  \psline[linewidth=1pt, arrows=<->](3.25,-9)(36.25,-9)
  \psline[linewidth=1pt, arrows=<->](36.25,-9)(39.75,-9)

  \psline[linewidth=0pt,linestyle=solid](0.25,-9)(0.25,0)
  \psline[linewidth=0pt,linestyle=solid](3.25,-9)(3.25,0)
  \psline[linewidth=0pt,linestyle=solid](36.25,-9)(36.25,0)
  \psline[linewidth=0pt,linestyle=solid](39.75,-9)(39.75,0)

  \rput(2,-11){$C_1$}
  \rput(21,-11){$C_2$}
  \rput(38,-11){$C_3$}
   \end{pspicture*}
\end{center}
Then we have
\[
g(\Psi_n) = \max\{g(C_1),g(C_2),g(C_3)\}
\]
Note that $g(C_1)=g(C_3) \leq p_1-1$ and $g(C_2) \leq \psi(n)-2\alpha$.
Since $\alpha \geq p_3-(p_1-1)$, we have 
$$g(C_2) \leq \psi(n)-2(p_3-(p_1-1))$$
Therefore, we have 
\[
g(\Psi_n) \; \leq \; \max\{p_1-1, \deg(\Psi_n)-2(p_3-(p_1-1))\}
\]

\citem{Claim 2:} $U\;>\;p_1-1.$\\
Note that
\begin{eqnarray*}
&&U - (p_1-1) \\
        &=& 2n(\frac{1}{p_1}+\frac{1}{p_2}+\frac{1}{p_3})-\deg(\Psi_n)-(p_1-1) \\
        &=& 2(p_1p_2+p_2p_3+p_3p_1)-(p_1p_2p_3-(p_1-1)(p_2-1)(p_3-1))-(p_1-1)\\
        &=& 2(p_1p_2+p_2p_3+p_3p_1)- p_1p_2p_3 +(p_1-1)(p_2p_3-p_2-p_3)   \\
        &=&  2(p_1p_2+p_2p_3+p_3p_1)- p_1p_2p_3 + p_1p_2p_3-p_1p_2-p_1p_3-p_2p_3+p_2+p_3 \\
        &=&  2(p_1p_2+p_2p_3+p_3p_1) -p_1p_2-p_1p_3-p_2p_3+p_2+p_3 \\
        &=&    p_1p_2+p_2p_3+p_3p_1 +p_2+p_3 \\
        &>& 0
\end{eqnarray*}
\citem{Claim 3:} $U \;>\; \deg(\Psi_n)-2(p_3-(p_1-1)).$\\
Note that
\begin{eqnarray*}
&&U - \left(\deg(\Psi_n)-2(p_3-(p_1-1))\right) \\
        &=& 2n(\frac{1}{p_1}+\frac{1}{p_2}+\frac{1}{p_3})-2\deg(\Psi_n)+2(p_3-(p_1-1)) \\
                      &=& 2(p_1p_2+p_2p_3+p_3p_1)-2(p_1p_2p_3-(p_1-1)(p_2-1)(p_3-1))+2p_3-2(p_1-1)\\
                      &=& 2(p_1p_2+p_2p_3+p_3p_1)-2(p_1p_2+p_2p_3+p_3p_1-p_1-p_2-p_3+1)+2p_3-2(p_1-1)\\
                      &=& 2p_1+2p_2+2p_3-2 +2p_3-2p_1+2\\
                      &=& 2p_2+4p_3 \\
                      &>& 0
\end{eqnarray*}

\end{clists}
\end{proof}

\end{document}